\documentclass[11pt]{amsart}

\usepackage[utf8]{inputenc}
\usepackage[english]{babel}
 
\usepackage{geometry}
\usepackage{anysize}
\usepackage{amsfonts}
\usepackage{amssymb}
\usepackage{amsmath}
\usepackage{amsthm}
\usepackage{tikz-cd}
\usepackage{amscd}
\usepackage[all]{xy}
\usepackage{enumerate}
\usepackage{multirow}
\usepackage{latexsym} 
\usepackage{comment}
\usepackage{color}
\usepackage{colonequals}
\usepackage{mathrsfs}

\usepackage{epsfig}
\usepackage{tikz}
\usepackage{tikz-cd}

\definecolor{darkgreen}{rgb}{0,0.5,0}
\usepackage[
        draft=false,
        colorlinks, citecolor=darkgreen,
        pdfauthor={Filippo F. Favale and Sonia Brivio},
        pdftitle={Coherent systems on nodal curves},
        linktocpage        
]{hyperref}

\usepackage[alphabetic,backrefs,lite]{amsrefs} 

\usepackage[all]{xy} 

\usepackage{enumerate} 

\newfont{\sheaf}{eusm10 scaled\magstep1}

\newtheorem{theorem}{Theorem}[section]
\newtheorem{proposition}[theorem]{Proposition}
\newtheorem{definition}[theorem]{Definition}

\newtheorem{corollary}[theorem]{Corollary}
\newtheorem{lemma}[theorem]{Lemma}
\newtheorem{remark}[theorem]{Remark}

\DeclareMathOperator{\Img}{Im}

\DeclareMathOperator{\Ker}{Ker}
\DeclareMathOperator{\Supp}{Supp}

\DeclareMathOperator{\Ext}{Ext}

\DeclareMathOperator{\Pic}{Pic}

\DeclareMathOperator{\rank}{rk}
\DeclareMathOperator{\Rk}{rk}
\DeclareMathOperator{\G}{Gr}
\DeclareMathOperator{\coker}{coker}

\newcommand{\Q}{\mathbb{Q}}

\newcommand{\R}{\mathbb{R}}
\newcommand{\PP}{\mathbb{P}}
\newcommand{\OO}{\mathcal{O}}

\newcommand{\w}{{\underline w}}
\newcommand{\wa}{({\underline w},\alpha)}
\DeclareMathOperator{\mrank}{\underline{rk}}
\DeclareMathOperator{\wdeg}{deg_{\underline{w}}}
\DeclareMathOperator{\wrank}{rk_{\underline{w}}}
\newcommand{\Gwa}{\mathcal{G}_{\underline{w},\alpha}}

\def\BOX{\hfill\lower.5\baselineskip\hbox{$\Box$}}

\newcommand{\Gr}{Gr}

\numberwithin{equation}{section}

\newcounter{nootje}
\renewcommand\check[1]
  {\marginpar{\tiny\begin{minipage}{20mm}\begin{flushleft}\thenootje : #1\end{flushleft}\end{minipage}}\addtocounter{nootje}{1}}
\begin{document}

\title[Coherent systems on curves of compact type]{Coherent systems on curves of compact type}

\author{Sonia Brivio}
\address{Dipartimento di Matematica e Applicazioni,
	Universit\`a degli Studi di Milano-Bicocca,
	Via Roberto Cozzi, 55,
	I-20125 Milano, Italy}
\email{sonia.brivio@unimib.it}

\author{Filippo F. Favale}
\address{Dipartimento di Matematica e Applicazioni,
	Universit\`a degli Studi di Milano-Bicocca,
	Via Roberto Cozzi, 55,
	I-20125 Milano, Italy}
\email{filippo.favale@unimib.it}

\date{\today}
 \thanks{
Both authors are partially supported by INdAM - GNSAGA.\\
}
 \maketitle
 
\begin{abstract}
Let $C$ be a polarized nodal curve of compact type.  In this paper we study coherent systems $(E,V)$ on $C$ given by a depth one sheaf $E$ having rank $r$ on each irreducible component of $C$ and a subspace $V \subset H^0(E)$ of dimension $k$. Moduli spaces of stable coherent systems have been introduced by \cite{KN} and depend on a real parameter $\alpha$. We show that when $k \geq r$, these moduli spaces coincide  for  $\alpha$ big enough.  Then we deal with the case $k=r+1$: when the degrees of the restrictions of $E$ are big enough we are able to describe an irreducible component  of this moduli space by using the dual span construction. 
\end{abstract}

\section*{Introduction}
Coherent systems on  smooth curves can be seen  as the generalisation of classical linear systems. They were studied first, under different names, by Bradlow (\cite{BD}), Bertram (\cite{B}) and Le Potier (\cite{Le}). They are closely related to higher rank Brill-Noether theory: for relevant results on this argument one can see for example \cite{Br}, \cite{N} and \cite{BGMN}. Moreover, coherent systems have been useful tool in studying theta divisors and the geometry of moduli spaces of vector bundles on curves (see for instance, some results of the authors: \cite{Bri1}, \cite{Bri2},  \cite{BF1} and  \cite{BV}). Nevertheless, they are already interesting by themselves since a notion of stability can be defined, depending on a real parameter $\alpha$. Varying $\alpha$ one gets a family of moduli spaces, providing examples of higher dimensional algebraic varieties with a rich and interesting geometry. For comparison of different notions of stability arising in moduli theory see for instance \cite{BB}.

Questions concerning emptiness or non emptiness, smoothness, irreducibility, and singularities have been  deeply studied by many authors. Among them, we can point out some relevant results in \cite{BGN} and \cite{BGMN}.  

Coherent systems can be defined even on a singular curve. A notion of semistability has been introduced, depending on a polarization $\underline{w}$ on the curve and a real parameter $\alpha$, and coarse moduli spaces can be constructed as well as in the smooth case (see \cite{KN}). Nevertheless, there has been little work in the singular case, (for example see \cite {B1} and \cite{B2}).  The situation becomes much better in the nodal case. In fact, many results of \cite{BGMN} and \cite{BGN} have been extended to irreducible nodal curves by Bhosle in \cite{Bho}. 

In this paper we start the study of coherent systems on a reducible nodal curve $C$ of compact type. By a coherent system on $C$ we mean a pair $(E,V)$, given by a depth one sheaf $E$ on $C$ and by a subspace $V \subseteq H^0(E)$. Fix a polarization $\w$ on the curve $C$, then we can   define   $\w$-rank and  $\w$-degree of $E$ (denoted as $\wrank(E)$ and $\wdeg(E)$). For any $\alpha \in {\mathbb R}$,  the  notion of $\wa$-stability has been  defined, see \cite{KN}.   

A pair $(E,V)$ is called generated (resp. generically generated) if the map of evaluation of global sections of $V$ is surjective (resp. generically surjective), see
Section \ref{generated} for details.
We focus our attention on  generated pairs  $(E,V)$ on the curve $C$   of multitype  $({\underline r}, d, k)$: 
i.e. $E$ has rank $r$ on each irreducible component of $C$, $\wdeg(E) = d$  and $ \dim V= k \geq r$. 
We denote by ${\mathcal G}_{\wa}({\underline r},d,k)$ the  moduli space parametrizing families of $\wa$-stable coherent systems as above. As in the  case  of smooth curves and of irreducible nodal curves, we have the following results, which are proved in   Proposition \ref{PROPO:gen}, Theorem \ref{THM:A} and Corollary \ref{COR:A} and are summarized as follows. 
\hfill\par
{\bf Theorem A}
\it
{Let $C$ be  a nodal  reducible  curve of compact type and let ${\underline w}$ be a polarization on it. 
Let $r \geq 1$, $ d \geq 0$ and $k \geq r$ integers.
There exists $\alpha_l \in {\mathbb R}$, depending on ${\underline w}$, $r$, $d$, $k$ and $p_a(C)$,  such that the moduli spaces 
${\mathcal G}_{\wa}({\underline r},d,k)$ coincide for $\alpha > \alpha_l$. 
Moreover, for any $\alpha > \alpha_l$, any $\wa$-stable $(E,V)$ is generically generated.} 
\hfill\par
\rm
Let  $(E,V)$ be a coherent system on the curve $C$, we can define in a natural way, coherent systems $(E_i,V_i)$ which  are  the restrictions of $(E,V)$ to each irreducible component $C_i$. As in the case of $\w$-stability for depth one sheaves on $C$ (see \cite{T2} and \cite{BF2}), $\alpha$-stability on the restrictions does not imply, in general, $\wa$-stability.  Nevertheless, when $k$ and $r$ are coprime, we give a sufficient condition in order to ensure that
$\alpha$-stability of restrictions implies $\wa$-stability: this is proved in Theorem \ref{THM:stab}.

In the second part of this paper we will concentrate on  coherent systems with $k=r+1$. In the case of smooth curves, they have been studied by using the dual span construction, which was introduced by Butler in \cite{But}. For Petri curves, it is almost completely known when   such moduli spaces are non empty (\cite{BBN}). In this case, for any $r \geq  2$, $d >0$, the moduli space ${\mathcal G}_{\alpha}(r,d,r+1)$ is birational to the moduli space ${\mathcal G}_{\alpha}(1,d,r+1)$  for large $\alpha$ (\cite{BGMN}).

We generalize  dual span  costruction to
generated coherent systems $(L,W)$ on  nodal a curve of compact type $C$ where $L$ is a line bundle. More precisely, we assume that $C$ has $\gamma$ irreducible components $C_i$ of genus $g_i \geq 2$. For any $(d_1,\dots, d_{\gamma}) \in {\mathbb N}^{\gamma}$, we can consider the subvariety
$$X_{d_1,\dots,d_{\gamma}} \subset {\mathcal{G}}_{\wa}({\underline 1},d,r+1),$$
parametrizing all coherent systems $(L,W)$ where  $L$ is a  line bundle whose restriction on $C_i$ has degree $d_i$. 
The first result for this part is Theorem \ref{Main1} which is summarized as follows.
\hfill\par
{\bf Theorem B}
\it
Under the hypothesis of Theorem A. 
If $d_i \geq max(2g_i +1,g_i +r)$ and $d= \sum_{i=1}^{\gamma} d_i$, then,  for $\alpha$ big enough, the closure
$\overline{X_{d_1,\dots,d_{\gamma}}} \subset {\mathcal G}_{\wa}({\underline 1},d,r+1)$
is an irreducible component of  dimension equal to the Brill-Noether number $\beta_C(1,d,r)$.
Any $(L,W) \in X_{d_1,\dots,d_{\gamma}}$ is a smooth point of the moduli space. 
\rm\hfill\par

Then by applying the  dual span construction to  coherent systems  in  
$X_{d_1,\dots,d_{\gamma}}$ we obtain the main results  of the second part of the paper. These are Theorem \ref{main2} and Theorem \ref{THM:B}  and are summarized  as follows. 
\hfill\par
{\bf Theorem C}
\it Under the hypothesis of Theorem B. 
For $\alpha$ big enough, there exists an irreducible component $Y_{d_1,\dots,d_{\gamma}} \subset {\mathcal{G}}_{\wa}({\underline r},d,r+1)$    which is birational to $X_{d_1,\dots,d_{\gamma}}$, with  dimension equal to the Brill-Nother number $ \beta_C(r,d,r+1)$. 
\hfill\par
We have the following commutative diagram:
$$
\xymatrix@C=50pt{
\overline{X_{d_1,\cdots,d_{\gamma}}}\ar@{-->}[r]^-{\mathcal{D}}\ar@{-->}[d]_{\pi_1} &     Y_{d_1,\dots,d_\gamma}\ar@{-->}[d]^{\pi_2} \\
  \Pi_{i=1}^{\gamma} \mathcal{G}_{C_i,\alpha}(1,d_i,r+1)\ar@{-->}[r]_-{\Pi D_i} & 
  \Pi_{i=1}^{\gamma} \mathcal{G}_{C_i,\alpha}(r,d_i,r+1)
}
$$
where 
$\mathcal{G}_{C_i,\alpha}(s,d_i,r+1)$  is  the moduli space of $\alpha$-stable coherent systems of type $(s,d_i,r+1)$ on the curve $C_i$, 
$\mathcal D$  and $D_i$ are  the map sending  a coherent system 
 to its dual span and the vertical maps are restrictions to the components of $C$. 
Finally,  the maps $\pi_1$ and $\pi_2$ are dominant. 
\hfill\par
\rm 
The paper is organized as follows. In Section 1 we recall  basic properties of nodal curves and depth one sheaves. In Section 2 we introduce  the notion of coherent system,  $\wa$-stability and we recall some results concerning their moduli spaces. In Section 3 we focus on generated coherent systems  of  multitype $({\underline r},d,k)$ and we prove Theorem A.  In Section 4, by using dual span construction, we produce  $\wa$-stable coherent systems of multitype $({\underline r},d,r+1)$. Finally, in  Section 5 we prove  the results stated in Theorem B and Theorem C.   
\hfill\par

\section{Nodal reducible curves and depth one sheaves}
\label{nodal curves}
In this paper we will consider connected reduced and reducible curves $C$ over the complex numbers which are {\it nodal}, i.e. complete algebraic curves whose singularities are at most {\it  ordinary double points}.
We recall that a connected nodal curve is said of {\it compact type} if every irreducible component of $C$ is smooth and  its  dual graph  is a tree. For theory of nodal curves see \cite{ACG}.
We will always assume that $C$ is a nodal curve of compact type and that each irreducible component $C_i$ of $C$ is smooth of genus $g_i\geq 2$.
If we denote by $\gamma$ the numbers of irreducible components of $C$  and by $\delta$ the number of nodes of $C$, we have $\gamma = \delta + 1$. The normalization map  of $C$ is
$$ \nu \colon C^{\nu} \to C,$$
where $C^{\nu} = {\bigsqcup}_{i=1}^{\gamma} C_i$ and  $\nu$ induces an isomorphism $\Pic (C) \simeq \bigoplus_{i=1}^{\gamma}\Pic(C_i)$ between the Picard groups.
In particular, we will denote by   $\Pic^{(d_1,\dots,d_{\gamma})}(C)$ the subgroup  of line bundles  $L$  on $C$ whose restriction to $C_i$ is  
 in $\Pic^{d_i}(C_i)$. 
The arithmetic genus of $C$ is 
$$p_a(C) = 1 - \chi(\OO_C) = \sum_{i = 1}^{\gamma} g_i.$$

\noindent We recall that, since $C$ is nodal, then  it can be  embedded in a smooth projective surface $S$, see \cite{A79}.  Let $B$ be any subcurve of $C$. The complementary curve of $B$,   denoted by $C-B$,  is defined as the closure of $C\setminus B$ and  it is actually the difference of $C - B$ as divisors on $S$. We will denote by $\Delta_B$ the intersection of $B$ with its complementary curve, it is given by double points common to a component of $B$ and one of $C-B$. In particular, when $C_i$ is a component of $C$, $\Delta_{C_i}$ is given by double points on $C_i$. To simplify notations  we  set:  $\Delta_{C_i}= \Delta_i$,  $C-C_i=C_i^{c}$ and  $\delta_i=\#\Delta_i$.
For any subcurve $B$ of $C$ we have the following exact sequence:
\begin{equation}
\label{EXSEQ:CB}
0\to \OO_{C-B}(-\Delta_B) \to \OO_C \to \OO_B \to 0,
\end{equation}
from which we deduce
$$p_a(C) = p_a(B) + p_a(C-B) + deg(\Delta_B) -1.$$ 
In particular, when  $B=C_i$, we have:
\begin{equation}
\label{EXSEQ:CI}
0\to \OO_{C_i^{c}}(-\Delta_i) \to \OO_C \to \OO_{C_i} \to 0,
\end{equation}
which gives $p_a(C_i^c)=\sum_{j\not=i}g_j +1 -\delta_i$. 

We stress a useful fact that we will use a lot in the following.

\begin{remark}
Since $C$ is of compact type, we can find a component $C_i$ such $C_i^{c}$ is a connected curve of compact type too. Actually, it is enough to show that  there exists a component $C_i$  with   $\delta_i=1$, i.e. such that only one node of $C$ lies on $C_i$. Assume on the contrary that, for all $i=1\dots \gamma$,  we have $\delta_i\geq 2$. As  every node lies on  two components, we have $$\delta=\frac{1}{2}\sum_{i=1}^{\gamma}\delta_i\geq \frac{1}{2}\sum_{i=1}^{\gamma}2=\gamma.$$
But this is impossible as $\gamma = \delta+1$.
\end{remark}

The dualizing sheaf $\omega_C$ is an invertible sheaf. Moreover, for any subcurve $B$ of $C$ the dualizing sheaf ${\omega}_B$ is  invertible too and we have:
\begin{equation}
    {\omega_C}_{\vert B} = {\omega}_B \otimes \OO_B(\Delta_B),
\end{equation}
in particular for any component $C_i$ we have:
${\omega_C}|_{C_i}= {\omega}_{C_i} \otimes \OO_{C_i}(\Delta_i).$
\hfill\par

\begin{definition} A {\it polarization} on the curve $C$   is a vector ${\underline w}=  (w_1,\dots,w_\gamma) \in \Q^{\gamma}$, with   \begin{equation}
\label{polarization}
 0 < w_i < 1, \quad   \sum_{i=1}^{\gamma}w_i=1.
\end{equation}
\end{definition}
We will fix  an ample primitive invertible sheaf $\OO_C(1)$ on the curve $C$,  with  $a_i = \deg({\OO_{C}(1)}|_{C_i})$. It determines a  polarization  ${\underline w}$  
by  defining  $w_i= \frac{a_i}{\sum_{k = 1}^{\gamma}a_k}$. Note that since $\OO_C(1)$ is ample, $a_i \geq 1$ and $\gcd(a_1,\dots,a_{\gamma})= 1$ since $\OO_C(1)$ is primitive. 
\hfill\par
We recall the notion of depth one sheaves on nodal curves, for details see \cite[Chapter VII]{Ses}. 
A coherent sheaf $E$ on $C$ is said of {\it depth one\footnote{Different terms are used to refer to such sheaves. As $C$ is  a scheme of pure dimension $1$, this is equivalent to ask that $E$ is {\it pure of dimension $1$} or that $E$ is {\it torsion free}. }} if  $\dim F = \dim supp(F) = 1$ for every subsheaf $F$ of $E$. A coherent sheaf $E$ on $C$ is of depth one if the stalk of $E$ at the node $p=C_i\cap C_j$ is isomorphic to $\OO_p^a \oplus \OO_{q_i}^{b_i} \oplus \OO_{q_j}^{b_j}$, where $\nu^{-1}(p)=\{q_i,q_j\}$ and $\OO_{q_t}=\OO_{C_t,q_t}$. In particular, any  vector bundle  on $C$ is a sheaf of depth one. Let $E$ be a sheaf of depth one  on $C$, its restriction $E\vert_{C_i \setminus \Delta_i}$ is either zero or it is locally free;
moreover, any subsheaf of $E$ is of depth one too.
\hfill\par
Let $E$ be a sheaf of depth one on $C$,  we set 
\begin{equation}
\label{restriction/torsion}
E_i = E \otimes \OO_{C_i} /Torsion,
\end{equation}
which is said {\it the restriction  of $E$ modulo torsion} on the component $C_i$.  If $E_i$ is not zero,  we set  $r_i= \rank(E_i)$;  otherwise we set $r_i = 0$. 
We associate to $E$:
\begin{equation}
\label{multirank}
\mrank(E) = (r_1,\dots,r_\gamma),    
\end{equation}
which is said {\it the multirank} of $E$;
\begin{equation}
\label{w-rank}
\wrank (E) = \sum_{i = 1}^{\gamma}w_i r_i
\end{equation}
which is said the ${\underline w}$-{\it rank} of $E$;
\begin{equation}
\label{w-deg}
\wdeg E = \chi(E)-\wrank(E) \chi(\OO_C),
\end{equation}
which is said the ${\underline w}$-{\it degree} of $E$. 

Note that ${\underline w}$-rank and ${\underline w}$-degree are not  necessary integers. When $E$ is a vector bundle on $C$, i.e. it is locally as $\OO_C^r$,  then the ${\underline w}$-rank of $E$ is actually $r$ and the ${\underline w}$-degree of $E$  is an integer too. 
\hfill\par

\begin{lemma}
\label{LEM:chi}
Let $E$ be a depth one sheaf on $C$ and let $E_i$ be the restriction modulo torsion of $E$ to $C_i$. Then we have:
\begin{enumerate}
\item let  $\mrank(E) = (r_1,\dots,r_\gamma)$ and $r_M=\max(r_1, \cdots,r_{\gamma})$:
$$\sum_{i=1}^{\gamma} \chi(E_i) -r_{M}(\gamma -1) \leq \chi(E) \leq \sum_{i=1}^{\gamma} \chi(E_i);$$
\item{} if $\mrank(E) = (r,\dots,r)$: 
$$  \sum_{i=1}^{\gamma} \deg(E_i) \leq \wdeg(E)  \leq \sum_{i=1}^{\gamma} \deg(E_i) + r(\gamma-1);$$
\item if $E$ is locally free  of rank $r$, then we have:
$$\chi(E) = \sum_{i=1}^{\gamma} \chi(E_i) -r(\gamma-1) \quad  \wdeg(E)= \sum_{i=1}^{\gamma} \deg(E_i);$$
\item{} if $E$ is locally free and $h^0(E_i) =0$ for any $i=1,\cdots,\gamma$, then we also have $h^0(E)= 0$. 
\end{enumerate}
\end{lemma}
\begin{proof}
(1) We have an exact sequence (see \cite{Ses})
$$ 0 \to E \to \bigoplus_{i=1}^{\gamma} E_i \to T \to 0 $$
where $T$ is a torsion sheaf whose support in contained in the set of nodes of $C$.  We deduce that $\chi(E)=\sum_{i=1}^{\gamma} \chi(E_i) - \chi(T)$. Note that  $\chi(T) = l(T) \geq 0$, hence 
 $\chi(E) \leq \sum_{i=1}^{\gamma} \chi(E_i)$.  Let  $p \in C_i \cap C_j$ be a node, such that $\nu^{-1}(p)=\{q_i,q_j\}$. If   $E_p \simeq \OO_p^s \oplus \OO_{q_i}^{b_i} \oplus \OO_{q_j}^{b_j}$ with $0 \leq s \leq \min(r_i,r_j)$, $s+b_i= r_i$  and $s+b_j= r_j$, then $T_{p} \simeq {\mathbb C}^s$,  see \cite{Ses}.  This implies  that   $l(T) \leq r_{M}(\gamma -1)$ and we can conclude that 
 $\chi(E) \geq \sum_{i=1}^{\gamma} \chi(E_i) - r_M(\gamma -1)$. 
\hfill\par
(2) From the above sequence we obtain 
$\wdeg(E) = \wdeg( \bigoplus_{i=1}^{\gamma} E_i ) - l(T)$. We have:
$$ \wdeg\left(\bigoplus_{i=1}^{\gamma} E_i \right) =
\sum_{i=1}^{\gamma}(\deg(E_i) +r(1-g_i))
-r(1-p_a(C)) = \sum_{i=1}^{\gamma}\deg(E_i) +r (\gamma-1).$$
As $0 \leq l(T) \leq r(\gamma -1)$, the assertion follows. 
\hfill\par

(3) If $E$ is  locally free,  for any node $p \in C_i \cap C_j$, we have $E_p \simeq \OO_p^{r}$. This implies that $l(T) = r(\gamma -1)$ and the first claim follows. 
By definition we have:
$$ \wdeg(E) =  \chi(E) -\wrank(E) \chi(\OO_C) = \sum_{i=1} ^{\gamma} \chi(E_i)- r(\gamma-1) -r \chi(O_C).$$
Since $\chi(E_i)= \deg(E_i) + r(1-g_i)$ we obtain:
$\wdeg(E) = \sum_{i=1}^{\gamma}\deg(E_i)$. 
\hfill\par
(4) We prove the assertion by induction on the number $\gamma$ of irreducible components of $C$. 
If $\gamma=2$, then $C$ has two irreducible components and  a single node $p$. By tensoring \ref{EXSEQ:CI} with $E$ we have the exact sequence
$$0 \to E_1(-p) \to E \to E_2 \to 0.$$
If we pass to global sections we obtain
$$0 \to H^0(E_1(-p)) \to H^0(E) \to H^0(E_2) \to...$$
and $h^0(E) = 0$ since  $h^0(E_2) =h^0(E_1(-p)) = 0$.
\hfill\par
Assume now that  $C$ is a nodal curve with  $\gamma \geq 3$ irreducible components. As we have seen, there exists an irreducible component $C_i$ having a single node $p_{ij}$. We can consider the exact sequence:
$$ 0 \to O_{C_i^c}(-p_{ij}) \to O_C \to O_{C_i} \to 0,$$
tensoring with $E$ we obtain:
$$ 0 \to E|_{C_i^c}(-p_{ij}) \to E \to E_i \to 0,$$
passing to global sections: 
$$ 0 \to H^0(E|_{C_i^c}(-p_{ij})) \to H^0(E) \to H^0(E_i) \to ...$$
Notice that $p_{ij}$ is a smooth point for $C_i^c$, so $E|_{C_i^c}(-p_{ij})$ is locally free on the curve $C_i^c$.
Moreover we have:
  ${E|_{C_i^c}(-p_{ij})}_{\vert C_j} = E_j(-p_{ij})$ 
  and ${E|_{C_i^c}(-p_{ij})}_{\vert C_k} = E_k$ for $k\not=i,j$. 
The curve $C_i^c$ is a nodal connected curve of compact type  with $\gamma -1$ components,  by induction hypothesis we have $h^0({E|_{C_i^c}(-p_{ij})}) = 0$.  Since $h^0(E_i) = 0 $ too, this implies that $h^0(E) = 0$. 

\end{proof}

\begin{lemma}
\label{LEM:h1}
Let $L$ be a line bundle on $C$, let  $L_i$ be the restriction of $L$ to the component $C_i$ and $d_i= deg(L_i)$. Then
\begin{enumerate}
\item $L$ is ample if and only if $d_i >0$ for any $i$;
\item if $d_i\geq 2g_i$, then $h^1(L)=0$ and $L$ is globally generated; 
\item if $d_i\geq 2g_i+1$ then the restriction map $\rho_i:H^0(L)\to H^0(L_i)$ is surjective and $L$ is very ample.
\end{enumerate}
\end{lemma}
\begin{proof}
(1) See \cite[Lemma 2.15]{ACG}.
\hfill\par
(2) By \cite[Lemma 2.1]{CF},  in order to have $h^1(L)=0$ and that $L$ is globally generated, it is enough to prove that for any subcurve $B$ of $C$ we have $\deg(L|_B)\geq 2p_a(B)$. Let $B$ be a subcurve of $C$ and assume that $B$ is connected. 
Then $B=\bigcup_{k=1}^{\gamma_B}{C_{i_k}}$ is a curve of compact type so $p_a(B)=\sum_{k=1}^{\gamma_B}g_{i_k}$ and $\deg(L|_B)=\sum_{k=1}^{\gamma_B}d_{i_k}$ by Lemma \ref{LEM:chi}. Since we are assuming $d_i\geq 2g_i$ for all $i$, we have
$$\deg(L|_B)=\sum_{k=1}^{\gamma_B}d_{i_k}\geq 2\sum_{k=1}^{\gamma_B}g_{i_k}=2p_a(B).$$

\noindent Assume now that $B$ is not connected. Then $B$ is the disjoint union $B=\bigsqcup_{k=1}^{c}{B_k}$ of  connected curves $B_1,\dots, B_c$ which are of compact type. It is easy to see that  
$$\deg(L|_{B})=\sum_{k=1}^c\deg(L|_{B_k})\qquad p_a(B)=\sum_{k=1}^{c}p_a(B_k)-(c-1).$$
Then, since we have $\deg(L|_{B_k})\geq 2p_a(B_k)$, as we just proved, we have
$$\deg(L|_B)=\sum_{k=1}^{c}\deg(L|_{B_k})\geq 2\sum_{k=1}^{c}p_a(B_k)>2p_a(B).$$

(3) If we tensor the exact sequence \ref{EXSEQ:CI} with $L$ and consider the long exact sequence in cohomology we have
$$0\to H^0(L|_{C_i^{c}}(-\Delta_i))\to H^0(L)\xrightarrow{\rho_i} H^0(L_i)\to H^1(L|_{C_i^{c}}(-\Delta_i))\to H^1(L).$$
By the previous point we have that $h^1(L)=0$ so the surjectivity of $\rho_i$ is equivalent to $h^1(L|_{C_i^{c}}(-\Delta_i))=0$. Denote by $C'$ the curve $C_i^{c}$ (which is a finite disjoint union of connected curves of compact type) and by $L'$ its line bundle $L|_{C_i^{c}}(-\Delta_i)$. Note that for any $j\neq i$ we have
$$d_j' = \deg(L'|_{C_j}) \geq d_j -1$$
since at most one of the points of $\Delta_i$ lies in $C_j$ (as each connected component of $C'$ is of compact type). Since, by assumption, we have $d_j'\geq 2g_j$, then $h^1(L|_{C_i^{c}}(-\Delta_i))=0$ by (2). This implies that  $\rho_i$ is surjective.

Finally, the very ampleness of $L$ follows from \cite{CFHR}: in fact, by using the same arguments of (2), one can prove that for any subcurve $B$ of $C$ we have $\deg(L|_B)\geq 2p_a(B) +1$. 
\end{proof}

\section{Coherent systems on nodal curves}
\label{cohsystem}
Let $C$ be a nodal reducible curve as in Section \ref{nodal curves}. In this section we will recall the notion of coherent systems\footnote{Note that the authors of \cite{KN} use the term {\it Brill-Noether pairs}.}  on  the curve $C$, for details see \cite{KN}. 
A {\it coherent system}  on the curve $C$ is  given by a pair $(E,V)$, where $E$ is a depth one sheaf on $C$ and $V$ is a subspace of $H^0(E)$.  A {\it coherent subsystem}  $(F,U)$ of $(E,V)$ is a coherent system which consists of a subsheaf $F \subseteq E$ and a subspace  $U\subseteq V \cap H^0(F)$. We say that $(F,U)$  is a {\it proper}  subsystem if $(F,U) \not= (0,0)$ and $(F,U) \not= (E,V)$. 
A coherent system $(E,V)$ is said of {\it type} $(r,d,k)$ if $\wrank(E) = r$, $\wdeg E = d$ and $\dim V = k$; 
if the multirank of $E$ is $\mrank(E) =(r_1,\cdots,r_{\gamma})$ then it is said of {\it multitype} $((r_1,\cdots,r_{\gamma}),d,k)$.

\begin{definition}
A family of coherent systems parametrized by a scheme $T$ is given by a  triplet 
$({\mathcal E}, {\mathcal V}, \xi)$ where
\begin{itemize}
\item ${\mathcal E}$ is a sheaf on $C \times T$ flat over $T$ such that for any $t \in T$ the sheaf  $E_t = {\mathcal E}_{\vert C \times t}$ is of depth one;
\item ${\mathcal V}$  is a locally free sheaf on $T$ whose fiber at $t$ is $V_t$;
\item $\xi \colon {\pi}^* {\mathcal V} \to {\mathcal E}$ is  a map of sheaves, where $\pi \colon C \times T \to T$ is the projection, and, for any $t \in T$, the map  
$$\xi_{ t} \colon V_t \otimes \OO_{C \times t} \to E_t,$$
induces an injective map $H^0(\xi_t) \colon V_t \to H^0(E_t)$.
\end{itemize}
Two families $({\mathcal E}, {\mathcal V}, {\xi})$ and $({\mathcal F}, {\mathcal U}, \eta)$ are isomorphic if and only if there exists an invertible sheaf ${\mathcal L}$ on $T$  such that 
${\mathcal F} \simeq  {\mathcal E} \otimes {\pi}^*{\mathcal L}$,
${\mathcal U} \simeq {\mathcal V} \otimes {\mathcal L}$ and 
$\eta = \xi \otimes {\pi}^*{\mathcal L}$.
\end{definition}

\begin{remark}
\label{families} \rm
Let $({\mathcal E},{\mathcal V},\xi)$ be a family of coherent systems parametrized by  a connected scheme $T$. Note that   the restriction 
${\mathcal E}|_{(C_i \setminus \Delta_i) \times T}$ is flat over $T$ too, so we have  that 
$$\rank({E_t}|_{C_i \setminus \Delta_i}) = r_i, \quad  \forall t \in T.$$ 
This implies that all coherent systems of the family have the same multitype.
Moreover, the set of $t \in T$ such that $E_t$ is locally free is an open subset of $T$. 
Finally, if ${\mathcal E}$ is  locally free, then ${\mathcal E}|_{C_i \times T}$ is flat over $T$ so it gives a family of vector bundles on the curve $C_i$ of rank $r_i$ and degree $d_i$. 
\end{remark}

We recall the notion of $\w$-slope for depth one sheaves on $C$ and  the definition of $\w$-stability. 

\begin{definition} Let $E$ be a depth one sheaf on $C$. For any polarization $\w$ on $C$  we define the 
$\w$-slope of $E$ as: 
$$\mu_{\w}(E)= \frac{\chi(E)}{\wrank(E)}=\chi(\OO_C)+\frac{\wdeg(E)}{\wrank(E)}.$$
$E$ is said $\w$-semistable (respectively $\w$-stable) if  for any proper subsheaf $F$  of $E$ we have $\mu_{\w}(F) \leq \mu_{\w}(E)$ (respectively $<$).
\end{definition}

The notion of $\wa$-slope and $\wa$-stability for coherent systems is defined as follows. 

\begin{definition} Let $(E,V)$ be a coherent system of type 
$(\wrank(E),\wdeg(E),k)$ on the curve $C$.  
For any positive $\alpha \in \R$ and for any polarization ${\underline w}$ on the curve $C$,  we define the  
$\wa$-slope of $(E,V)$:
$$\mu_{{\underline w},\alpha} (E,V) = \frac{\wdeg(E)}{\wrank(E)} + \alpha \frac{k}{\wrank(E)}=\frac{\chi(E)}{\wrank(E)}- \chi(\OO_C)+\alpha\frac{k}{\wrank(E)}.$$
\end{definition}

\begin{definition}
A coherent system $(E,V)$ is said $\wa$-semistable (resp. stable) if for any proper coherent subsystem $(F,U)$ we have:
$$ \mu_{{\underline w},\alpha} (F,U) \leq \mu_{{\underline w},\alpha} (E,V)  \ \  (resp. <).$$
\end{definition}

Fix $(r,d,k)$ with $r,d\in\mathbb{R}$, $r>0$, $k\in \mathbb{N}$ and $\alpha \in {\mathbb Q}$ positive. In \cite{KN} it is proved that there exists a projective scheme 
${\tilde{\mathcal G}}_{\underline{w},\alpha}(r,d,k)$  which is a coarse moduli space for families of
$\wa$-semistable coherent systems of type
$(r,d,k)$ on the curve $C$. 
Moreover, let  $\Gwa(r,d,k)$ denote the subscheme parametrizing $\wa$-stable coherent systems, it is an open subscheme of  $\tilde{\mathcal{G}}_{\underline{w},\alpha}(r,d,k)$. 
As $C$ is a  reducible curve, these spaces are reducible too, different components correspond to  possible  multiranks $(r_1,..,r_{\gamma})$, see Remark \ref{families}.  We are interested in those components containing coherent systems arising from locally free sheaves of rank $r$. 
With this aim, we set  $\underline{r}=(r,\dots,r)$ and let   $\Gwa(\underline{r},d,k)$ denote the subscheme of $\Gwa(r,d,k)$ parametrizing families of $({\underline w},\alpha)$-stable coherent systems of multitype $(\underline{r},d,k)$. 
We will denote by 
$$\Gwa'(\underline{r},d,k) \subset \Gwa(\underline{r},d,k)$$
the open subset corresponding to $\wa$-stable coherent systems $(E,V)$ with $E$ locally free. 
 
We have the following  fundamental result:
\begin{theorem}
Let $(E,V)$ be a coherent system which is $\wa$-stable and let $\Lambda \in \Gwa(r,d,k)$ be the corresponding point.
\begin{enumerate}
\item The Zariski tangent space  of $\Gwa(r,d,k)$ at the point $\Lambda$ is isomorphic to $\Ext^1(\Lambda,\Lambda)$;
\item if $\Ext^2(\Lambda,\Lambda) = 0$, then $\Gwa(r,d,k)$ is smooth of dimension $\dim \Ext^1(\Lambda,\Lambda)$ at the point $\Lambda$;
\item for every irreducible component $S$ of $\Gwa(r,d,k)$ through $\Lambda$ we have: 
$$ \dim \Ext^1(\Lambda,\Lambda) - \dim \Ext^2(\Lambda,\Lambda) \leq \dim S \leq \dim \Ext^1(\Lambda,\Lambda).$$
\end{enumerate}
\end{theorem}

This theorem has been proved in the case of smooth curves in \cite{He}. Actually, the machinery  introduced by the author  in order to prove the result also works for arbitrary reduced nodal curves. This has been also noted in \cite{Bho} (where the author is interested in the irreducible case).

When  $r,k\in \mathbb{N}$ and $d\in \mathbb{Z}$, as in the smooth case, we can define the {\it Brill-Noether number} :
\begin{equation}
\beta_C(r,d,k) = r^2(p_a(C) -1) + 1 - k(k-d +r(p_a(C)-1)).
\end{equation}
If $\Lambda \in \Gwa(r,d,k)$ corresponds to a coherent system $(E,V)$  with $E$ locally free, then we can define the {\it Petri map} $\mu_{E,V}$ of $(E,V)$ as follows:
\begin{equation}
\label{Petri map}
\mu_{E,V} \colon V \otimes H^0({\omega}_C \otimes E^*) \to H^0({\omega}_C \otimes E \otimes E^*)
\end{equation}
which is given by multiplication of global sections. 
For coherent systems $(E,V)$ with $E$ locally free, we have the following result:
\begin{proposition}
\label{Petritheorem}
Let $\Lambda \in \Gwa(r,d,k)$ corresponding to a coherent system $(E,V)$ with $E$ locally free. Then, if the Petri map of $(E,V)$ is injective, $\Lambda$ is a smooth point of $\Gwa(r,d,k)$ and the dimension of  $\Gwa(r,d,k)$ at $\Lambda$ is given by the Brill-Noether number. 
\end{proposition}
This result has been proved for smooth curves in \cite[Proposition 3.10]{BGMN} and has been generalized to a nodal irreducible curve in \cite[Proposition 3.7]{Bho}. Actually, as previously noted, the arguments involved in the proof of this proposition still works for a reduced nodal curve too. 

\section{Generated coherent systems on nodal curves}
\label{generated}
Let $C$ be a nodal reducible curve as in  Section \ref{nodal curves} with $\gamma$ components.

\begin{definition}
A coherent system $(E,V)$ on  $C$ of type $(r,d,k)$ is said {\it generated} if the evaluation map of global sections 
$$ev_V \colon V \otimes O_C \to E$$
is surjective. It is said {\it generically generated} if either it is generated  or $\coker ev_V$  is a sheaf  whose support is $0$-dimensional. 
\end{definition}

Assume that $(E,V)$ is a coherent system on $C$. For any  connected subcurve $B$ of $C$  we can define the restriction of $(E,V)$ to $B$  as follows. From the  exact sequence
$$0\to \OO_{C-B}(-\Delta_B) \to \OO_C \to \OO_B \to 0$$
by tensoring with $E$,  we have a surjective map too $E \to E \otimes \OO_B$ which is actually  the restriction map.  
Then, if we set    $E_B=E\otimes {\OO_B} / torsion$, we have 
a surjective map too  $ E \to E_{B}$
which induces the following  map of global sections:
$$ \rho_B \colon H^0(E)\to H^0(E_B).$$
We define $V_B$ as the  image of $V$ by the map   $\rho_B$. 
Then $(E_B,V_B)$ is a coherent system on the subcurve $B$.  Notice that when  $E$ is a vector bundle then $E_B=E\otimes \OO_B$ and it is a vector bundle too on $B$.
\begin{definition}
We will call $(E_B,V_B)$  the restriction of $(E,V)$ to the subcurve $B$. When $B=C_i$, to simplify  notations, we will denote it by $(E_i,V_i)$.  
\end{definition}

\begin{lemma}
\label{LEM:GLOBGEN}
Let  $(E,V)$ be  a generated (respectively generically generated) coherent system on $C$. If $B$ is a connected subcurve of $C$, then $(E_B,V_B)$ is generated  (respectively generically generated) too. 

\end{lemma}

\begin{proof}
Consider the evaluation map  $ev_V:V\otimes \OO_C\to E$. As both  map $V\to V_B$ and $\OO_C\to \OO_B$ are restriction maps 
we have a commutative diagram
$$
\xymatrix{
V\otimes \OO_C \ar[r]^-{ev_V} \ar@{->>}[d] & E \ar@{->>}[r] \ar@{->>}[d] & \coker ev_V \ar@{->>}[d] \\
V_B\otimes \OO_B \ar[r]^-{ev_{V_B}} & E_B \ar@{->>}[r] & \coker ev_{V_B} 
}
$$
where vertical maps are surjective. 
If $(E,V)$ is generated, then $ev_V$ is surjective and  $ev_{V_B}$ is surjective too, so 
 $(E_B,V_B)$ is generated.
If $(E,V)$ is generically generated (but not generated),  by the above diagram, $\dim \Supp(\coker e_V) = 0$ implies  $\dim \Supp(\coker e_{V_B}) = 0$ too.
\end{proof}

From now on, we will restrict our attention to coherent systems on $C$  of  multitype $(\underline{r},d,k)$, where we set  ${\underline r} = (r,\cdots,r)$.

\begin{remark}
\label{REM:gen}
Let $(E,V)$ be a coherent system of multitype $(\underline{r},d,k)$. If it is generically generated, then by Lemma 
\ref{LEM:GLOBGEN},
$(E_i,V_i)$ is a generically generated coherent system  on the curve $C_i$ with  $d_i \geq 0$ and $\dim V_i \geq r$. Hence, by  Lemma \ref{LEM:chi}, we also have  $d\geq 0$   and $k \geq r$,  as $\dim V \geq \dim V_i$. 
\end{remark}

The following property  is a generalization of  \cite[Proposition 4.4]{BGMN}  and  of \cite[Corollary 3.15]{Bho}.

\begin{proposition}
\label{PROPO:gen}
Let $C$ be  a nodal curve as in Section \ref{nodal curves} and  let  ${\underline w}$ be  a polarization on it. Fix integers $r \geq 1$, $d \geq 0$ and $k \geq r$. There exists $\alpha_g \in {\mathbb R}$ such that, for any $\alpha  \geq  \alpha_g$, any coherent system $(E,V)$ of type $(\underline{r},d,k)$,  which is $\wa$-semistable,  is generically generated and the kernel of the evaluation map $ev_V$ has no global sections.  
\end{proposition}
\begin{proof}
Let $(E,V)$ be a coherent system  which is $\wa$-semistable.
Assume that   $(E,V)$ is  not generically generated. We denote by $G$ the image of the  evaluation map $ev_V$. Then $G$ is a depth one subsheaf of $E$,
 with $\Rk(G_i) = r_i \leq r$  for any $i$,  which  fits into the exact sequence
$$ 0 \to G \to E \to \coker ev_V \to 0$$
and $\dim \Supp(\coker ev_V) = 1$. 
There exists a component $C_j$ such that $C_j\subseteq \Supp(\coker ev_V)$,  so 
we have $\dim \Supp(\coker ev_{V_j}) = 1$ too and   $r_j<r$. This implies  $\wrank(G)<\wrank(E) = r$.
Consider the coherent system $(G,V)$,  it  is a proper subsystem of $(E,V)$ and it is also generated by construction. From the $\wa$-semistability of $(E,V)$ we have
$$\mu_{\wa}(G,V) \leq \mu_{\wa}(E,V),$$
equivalently
$$\frac{\wdeg(G)}{\wrank(G)} + \alpha \frac{k}{\rank(G)} \leq 
\frac{d}{r} + \alpha \frac{k}{r}\qquad \mbox{i.e.}\qquad \alpha k \left(\frac{1}{\wrank(G)}-\frac{1}{r}\right) \leq
\frac{d}{r}-\frac{\wdeg(G)}{\wrank(G)}.$$
Note that the coefficient of $\alpha$ is positive, since $\wrank(G)<r$, so we can write $$\alpha\leq \frac{d\wrank(G)-r\wdeg(G)}{k(r-\wrank(G))}.$$
We recall that for any $i$ we have,  see Section \ref{nodal curves}, that $w_i = \frac{a_i}{\sum_{m=1}^{\gamma} a_m}$, with $a_m \geq 1$.
Note that, since $r_j < r$, we have  
$$r-\wrank(G)=\sum_{i=1}^{\gamma}w_i(r-r_i)\geq w_j(r-r_j)\geq w_j = \frac{a_j}{\sum_{m=1}^{\gamma}a_m},$$
so
$$ \alpha \leq \frac{\sum_{m=1}^{\gamma}a_m}{ka_j}(d \wrank(G) - r \wdeg(G)).$$
 By Lemma \ref{LEM:chi}, we have
$$\wdeg(G) = \chi(G) -\wrank(G) \chi(\OO_C)  \geq \sum_{i=1}^{\gamma}
\chi(G_i) -r_M\delta - (1 - p_a) \sum_{i=1}^{\gamma}w_ir_i $$
and we obtain:
$$ \alpha \leq \frac{\sum_{m=1}^{\gamma}a_m}{ka_j}\left( d \sum_{i=1}^{\gamma}w_ir_i - r\sum_{i=1}^{\gamma}\deg(G_i) +r \sum_{i=1}^{\gamma}r_i(g_i-1) +rr_{M}\delta -r(p_a-1) \sum_{i=1}^{\gamma}w_ir_i\right).$$
As $(G_i,V_i)$ is generated $\deg(G_i) \geq 0$.  Since  $r_i \leq r$ and $r_j < r$ we have:  $r_{M} \leq r$ and  $\sum_{i=1}^{\gamma}w_ir_i < r$. Finally,  $a_j \geq 1$ by construction,  so the above inequality become: 
$$\alpha < \frac{\sum_{m=1}^{\gamma}a_m}{k}(d r  +r^2 (p_a(C)-1) )= \alpha_g.$$
Hence the first claim is proved. 

Let $(E,V)$ be generically generated. Then we have the exact sequence
$$ 0 \to \Ker ev_V \to V \otimes O_C \to E \to \coker ev_V \to 0,$$
where  $\coker ev_V$ has $0$-dimensional support. 
Since $H^0(ev_V) \colon V \to H^0(E)$ is injective, then we have
$H^0(\Ker ev_V) = 0$.
\end{proof}

Note that in the proof of the above Theorem we have defined
\begin{equation}
\label{DEF:ALPHAG}
\alpha_g = \frac{\sum_{m=1}^{\gamma}a_m}{k}(d r  +r^2 (p_a(C)-1)).
\end{equation}
This number depends only on the arithmetic genus of 
$C$, the polarization $\w$ and the multitype $(\underline{r},d,k)$.

The following property  generalizes 
\cite[Proposition  4.5 (i)]{BGMN}.  

\begin{proposition}
\label{PROP:stronglyunstable}
Let $C$ be  a nodal curve as in Section \ref{nodal curves} and  let  ${\underline w}$ be  a polarization on it. 
Let $(E,V)$ be a generically generated coherent system of multitype  $(\underline{r},d,k)$. 
 If there exists a proper subsystem $(F,U)$ of $(E,V)$ 
 such that 
 \begin{equation}
 \label{ineq:strong}
 \frac{\dim U}{\wrank(F)} > \frac{\dim V}{\wrank(E)},
 \end{equation}
 then $(E,V)$ is not $\wa$-semistable  for $\alpha > k\alpha_g$.  
\end{proposition}
\begin{proof}
Let $(E,V)$ be a generically generated coherent system which is $\wa$-semistable. 
Assume that there exists a proper subsystem  $(F,U)$  of $(E,V)$ satisfying \ref{ineq:strong}. Since $U \subseteq V$, then we have $\wrank(F) < \wrank(E)= r$.
We can assume that $(F,U)$ is generated, otherwise we can consider the subsystem 
$(\Img ev_U, U)$ of $(E,V)$ which satisfies \ref{ineq:strong} too.
\hfill\par
Let $h = \dim U$. Since $(E,V)$ is $\wa$-semistable we have
$$\frac{\wdeg(F)}{\wrank(F)}+ \alpha \frac{h}{\wrank(F)} \leq  \frac{d}{r}+\alpha\frac{k}{r}.$$
This is equivalent to:
$$\alpha \left( \frac{h}{\wrank(F)}-\frac{k}{r}\right) \leq \frac{d}{r} - \frac{\wdeg(F)}{\wrank(F)}.$$
Then by \ref{ineq:strong} we  get
$$
\alpha \leq \frac{d\wrank(F) - r \wdeg(F)}{
hr - k\wrank(F)}.
$$
Note that   $k\wrank(F)$ is a rational number,  we denote by  $\lfloor k\wrank(F)\rfloor$ its integral part and by $\{k\wrank(F)\}$ its fractional part. 
By \ref{ineq:strong} we get $hr \geq \lfloor k\wrank(F)\rfloor +1$, which  implies that:
$$ hr - k\wrank(F) \geq 1 - \{k\wrank(F)\}.$$
We recall that for any $i$ we have,  see Section \ref{nodal curves}, that $w_i = \frac{a_i}{\sum_{m=1}^{\gamma} a_m}$, with $a_m \geq 1$. 
So we have $\{k\wrank(F)\} = \frac{b}{\sum_{m=1}^{\gamma}a_m}$ with 
$0 \leq b \leq \sum_{m=1}^{\gamma}a_m -1$.
This allows us to prove the bound
$$hr - k\wrank(F) \geq \frac{1}{\sum_{m=1}^{\gamma}a_m}.$$
Hence we obtain
$$ \alpha \leq  \left(\sum_{m=1}^{\gamma}a_m\right) (d \wrank(F) - r \wdeg(F)).$$
Since $(F,U)$ is generated we can proceed as in the proof of proposition \ref{PROPO:gen} in order to obtain
$$\alpha \leq \left(\sum_{m=1}^{\gamma} a_m\right)(dr +r^2(p_a(C) -1))= k\alpha_g.$$

\end{proof}
As in the  case of smooth curves and of irreducible  nodal curves, this gives us  a  necessary condition for $\wa$-semistability.    

\begin{definition}
\label{propertystar}
Let $(E,V)$ be a coherent system of multitype  $({\underline r},d,k)$.
We say that  $(E,V)$  satisfies property $(\star)$ (property $(\star')$ respectively) if for any proper coherent subsystem $(F,U)$ of $(E,V)$ we have either $(\star_1)$ or $(\star_2)$ ($(\star_1)$ or $(\star_2')$ respectively) where
\begin{description}
\item [$(\star_1)$] $\frac{\dim U}{\wrank(F)} < \frac{\dim V}{\rank(E)}$
\item [$(\star_2)$] $\frac{\dim U}{\wrank(F)} = \frac{\dim V}{\rank(E)} \mbox{ and } 
\frac{\wdeg(F)}{\wrank(F)} < \frac{\wdeg(E)}{\rank(E)}$
\item [$(\star_2')$] $\frac{\dim U}{\wrank(F)} = \frac{\dim V}{\rank(E)} \mbox{ and } 
\frac{\wdeg(F)}{\wrank(F)} \leq \frac{\wdeg(E)}{\rank(E)}$
\end{description}
\end{definition}
\begin{remark}
 Under the hypothesis  of Proposition \ref{PROP:stronglyunstable}, we have the following  properties:
 \hfill\par
 \begin{enumerate}
\item{} if $(E,V)$ is a coherent system which is $\wa$-stable ($\wa$-semistable respectively) for any $\alpha > k\alpha_g$, then $(E,V)$ satisfies property  $(\star)$ ($(\star')$ respectively). 
\hfill\par
\item{} if $(E,V)$ is a coherent system   which satisfies property $(\star)$ (property $(\star')$ respectively) and $E$ is ${\underline w}$-stable ($\w$-semistable respectively), then $(E,V)$ is $\wa$-stable ( $\wa$-semistable respectively) for any $\alpha > 0$.
 \end{enumerate}

\end{remark}

The following Theorem   generalizes the results of 
\cite[Proposition  4.5]{BGMN}  and \cite[Proposition 3.16]{Bho}.

\begin{theorem}
\label{THM:A}
Let $C$ be a nodal reducible curve as in Section \ref{nodal curves} and let $\underline{w}$ be a polarization on it. Let $(E,V)$ be a coherent system of multitype $({\underline r},d,k)$ as above. If $(E,V)$ is generically generated  and satisfies property $(\star)$ ($(\star')$ respectively), then $\forall \alpha> k\alpha_g$, $(E,V)$ is $\wa$-stable  ($\wa$-semistable respectively). 
\end{theorem}
\begin{proof}
Let $(E,V)$ be a generically generated coherent system of multitype $({\underline r},d,k)$ satisfying property $(\star)$.  Assume that there exists a proper coherent subsystem $(F,U)$ destabilizing $(E,V)$. 
 Then $F \subseteq E$ is a subsheaf of depth one,   $U \subseteq V$ with $\dim U = h \leq k$.  Since $(E,V)$ satisfies property $(\star)$, we
have  
$$\frac{h}{\wrank(F)} \leq  \frac{k}{r}.$$ 
\hfill\par
If equality holds, then we are in case $(\star_2)$ so $\frac{\wdeg(F)}{\wrank(F)} < \frac{d}{r}$. This implies
$$ \frac{\wdeg(F)}{\wrank(F)}  + \alpha  \frac{h}{\wrank(F)} < \frac{d}{r} + \alpha \frac{k}{r},$$
for any $\alpha >0$ and thus we have a contradiction. 
\hfill\par
    
So we are in case $(\star_1)$.  The inequality
\begin{equation}
\label{destab}
\mu_{\wa}(F,U)=\frac{\wdeg(F)}{\wrank(F)}+ \alpha \frac{h}{\wrank(F)} \geq \frac{d}{r}+\alpha\frac{k}{r}=\mu_{\wa}(E,V),
\end{equation}
is equivalent to
$$\alpha\left( \frac{k}{r}-\frac{h}{\wrank(F)}\right) \leq  \frac{\wdeg(F)}{\wrank(F)}-\frac{d}{r},$$
which gives
\begin{equation}
\label{mainineq}    
\alpha \leq  \frac{r\wdeg(F)-\wrank(F)d}{
k\wrank(F)-hr}.
\end{equation}
As in the proof of  Proposition \ref{PROP:stronglyunstable}, we denote by  $\lfloor k\wrank(F)\rfloor$ the integral part of $k\wrank(F)$ and  by $\{k\wrank(F)\}$ its fractional part. 
Since $\frac{h}{\wrank(F)} < \frac{k}{r}$, we get 
$$hr \leq 
\begin{cases} 
\begin{array}{cc}
\lfloor k\wrank(F)\rfloor -1 & \mbox{ if } \{k\wrank(F)\}= 0 \\
\lfloor k\wrank(F)\rfloor & \mbox{ if } \{k\wrank(F) \} \not= 0 
\end{array}
\end{cases}\Longrightarrow -hr\geq -\lfloor k\wrank(F)\rfloor.$$
This implies that
$ k\wrank(F)  -hr \geq  \{ k\wrank(F) \} = \frac{b}{\sum_{m=1}^{\gamma}a_m}$,
 with 
$1 \leq b \leq \sum_{m=1}^{\gamma}a_m -1$.
Hence we obtain 
$$k\wrank(F)  -hr \geq \frac{1}{\sum_{m=1}^{\gamma} a_m},$$
and Inequality \ref{mainineq}  becomes:
$$ \alpha \leq  \left(\sum_{m=1}^{\gamma}a_m\right)( r\wdeg(F) -d \wrank(F)) \leq 
\left(\sum_{m=1}^{\gamma}a_m\right)( r\wdeg(F)),$$
as  $\wrank(F) > 0$ and $d \geq 0$, since $(E,V)$ is generically generated, see Remark \ref{REM:gen}.  
By definition, we have:  $\wdeg(F) = \chi(F) - \wrank(F) \chi(O_C)$,
and by Lemma \ref{LEM:chi} (1) we have:
$$\chi(F) \leq \sum_{i=1}^{\gamma}\chi(F_i)=\sum_{i}^{\gamma}(\deg(F_i)+r_i(1-g_i)).$$
Since $F_i$ is a subsheaf of $E_i$, which is generically generated,  the quotient $E_i/F_i$ is generically generated too and so has non negative degree.  This implies that  $\deg(F_i) \leq \deg(E_i)$.  By Lemma \ref{LEM:chi} (2), $\sum_{i=1}^{\gamma}\deg(E_i) \leq \wdeg(E)= d$, so that  
$$
\chi(F) \leq d -\sum_{i=1}^{\gamma}r_i(g_i-1) \leq d,
$$
since $r_i \geq 0$  and $g_i \geq 2$.
Finally, we obtain:
$$\alpha \leq  \left(\sum_{m=1}^{\gamma}a_m\right)( rd +r \wrank(F)(p_a(C)-1)) \leq 
\left(\sum_{m=1}^{\gamma}a_m\right)( rd +r^2(p_a(C)-1)) = k \alpha_g.$$
\end{proof}

\begin{corollary}
\label{COR:A}
Let $C$ and $\w$ as in Theorem \ref{THM:A}. For any  $r \geq 1$, $d \geq 0$, $k \geq r$ integers, the moduli spaces ${\mathcal G}_{\wa}({\underline r},d,k)$ coincide for any $\alpha > k{\alpha_g}$. 
\end{corollary}

\begin{corollary}
In the hypothesis of Theorem \ref{THM:A}. If $(E,V)$ satisfies property $(\star)$, then each  restriction $(E_i,V_i)$ satisfies the condition $\dim V_i \geq w_i\dim V$.
\end{corollary}
\begin{proof}
Consider the subsheaf $G \subset E$ which is the kernel of the restriction map $\rho_i \colon E  \to E_i$:
$$ 0 \to G \to E \to E_i \to 0.$$
It is a depth one sheaf and $\wrank(G)= \wrank(E) - \wrank(E_i)= r(1 -w_i)$. Let $U \subset V$ be the kernel of the restriction map $\rho_i:V \to V_i$. Then $\dim U = \dim V - \dim V_i \geq 0$ and 
$U \subseteq H^0(G)$. Hence $(G,U)$  is a proper coherent subsystem of $(E,V)$. Since it satisfies property $(\star)$ we have
$$ \frac{\dim U}{\wrank(G)} \leq \frac{\dim V}{\wrank(E)},$$
which is equivalent to 
$$  r(\dim V -\dim V_i) \leq r(1-w_i)\dim V, $$
which gives us $\dim V_i \geq w_i \dim V$.
\end{proof}

Finally we have the following sufficient condition for $\wa$-stability. 
\begin{theorem}
\label{THM:stab}
Let $C$ be a reducible nodal curve as in Section \ref{nodal curves} and let $\underline{w}$ be a polarization on it. Let $(E,V)$ be a  coherent system of multitype $({\underline r},d,k)$ and denote with $(E_i,V_i)$ its restrictions to $C_i$. 
Assume that:
\begin{enumerate}
\item{} $(E,V)$ is generically generated;
\item $(E_i,V_i)$ is a coherent system of type $(r,d_i,k)$ on $C_i$;
\item $(E_i,V_i)$ is $\alpha$-stable  
for any $\alpha>d_i(r-1)$;
\item either $r$ and $k$ are coprime or $E$ is $w$-stable. 
\end{enumerate}
Then $(E,V)$ is $\wa$-stable $\forall \alpha>k \alpha_g$.
\end{theorem}
\begin{proof}
By Theorem \ref{THM:A} it is enough to prove that $(E,V)$ satisfies property $(\star)$. 
Let $(F,U)$ be a proper coherent subsystem  of $(E,V)$. First of all we  have to prove that:
\begin{equation}
\label{ineq:2}
\frac{\dim U}{\wrank(F)} \leq \frac{\dim V}{\wrank(E)}.
\end{equation}
We recall that $F \subseteq E$ is a subsheaf of depth one and $U \subseteq V$ with $\dim U = h \leq k$. If $h=0$ then \ref{ineq:2} is satisfied. So we can assume $h \geq 1$.  For any $i$ we  consider the restriction  $(F_i,U_i)$, it  is a coherent subsystem of $(E_i,V_i)$. In particular, since by assumption (2),  the restriction map $\rho_i|_V \colon V \to V_i$ is an isomorphism,   then  $\rho_i|_U \colon U \to U_i$ is  an isomorphism too. This implies  $\dim U_i = \dim U =h$.    Let $r_i= \rank(F_i)$,  as $\dim U_i \geq 1$, then $r_i \geq 1$.  Since $(E_i,V_i)$ is $\alpha$-stable for $\alpha > d_i(r-1)$, then, by \cite[Proposition 4.5]{BGMN}, it satisfies property $(\star)$, in particular:  
\begin{equation}
\label{ineq:1}
\frac{h}{r_i} \leq \frac{k}{r},
\end{equation}
equivalentely  $hr \leq kr_i$.
From the above inequality we deduce the following: 
$$hr = \sum_{i=1}^{\gamma}w_ihr \leq \sum_{i=1}^{\gamma}w_ikr_i = k \wrank(F), $$ 
which is equivalent to  \ref{ineq:2}. 
\hfill\par
Finally, if $r$ and $k$ are coprime,  then  both the inequality \ref{ineq:2} and \ref{ineq:1} are strict, so $(E,V)$ satisfies property $(\star_1)$ (and $(\star_2)$ cannot occur). If $r$ and $d$ are not coprime, by hypotesis, we have that $E$ is $\w$-stable so 
$$\frac{\wdeg(F)}{\wrank(F)}< \frac{\wdeg(E)}{\wrank(E)}.$$
This, with the inequality \ref{ineq:2} guarantees that $(E,V)$ satisfy $(\star)$ as claimed.
\end{proof}

\section{Construction of coherent system of type $(r,d,r+1)$}
Let $C$ be a nodal curve as in Section \ref{nodal curves}, with $\gamma$ irreducible components. 
Let $L \in \Pic^{(d_1,\cdots,d_{\gamma})}(C)$ be a  globally generated line bundle on $C$. 
 From Lemma \ref{LEM:chi} we have that 
$$\chi(L)=\sum_{i=1}^{\gamma} \chi({L_i})-\gamma+1 \quad 
deg_{\underline w}(L) = d = \sum_{i=1}^{\gamma}d_i.$$

\hfill\par

Let $r \geq 1$ and consider a subspace $W\subseteq H^0(L)$ of dimension $r+1$  such that the evaluation map 
$$ev_W :W\otimes \OO_C\to L$$
is surjective. Then $(L,W)$ is a generated coherent system of multitype $({\underline 1},d,r+1)$  on the curve $C$. Let $(L_i,W_i)$ be the restriction of $(L,W)$ to the component $C_i$.
By Lemma \ref{LEM:GLOBGEN},    $(L_i,W_i)$ is a generated coherent system on $C_i$ of type $(1,d_i,k_i)$. This implies $d_i \geq 0$, $\forall i = 1, \cdots, \gamma$. 
For any $i$, we define $R_i$ as follows
\begin{equation}
\label{Ri}
R_i = H^0(L|_{C_i^c}(-\Delta_i)) \cap W,
\end{equation}
where $\OO_{{C_i}^c}(-\Delta_i)$ is defined in the exact sequence \ref{EXSEQ:CI}.

\begin{proposition}
\label{Prop1}
Let $C$ be a nodal curve as in Section \ref{nodal curves}, with $\gamma$ irreducible components and ${\underline w}$ a polarization on it. 
Let  $L \in \Pic^{(d_1,\cdots, d_{\gamma})}(C)$ be  a line bundle  on the curve $C$.  Let $(L,W)$ be a generated coherent system  of multitype $({\underline 1},d,r+1)$ and  assume that $R_i = 0$ for any $i = 1, \cdots, \gamma$.
Then we have:
\begin{enumerate}
    \item{} $(L_i,W_i)$ is a generated coherent system on $C_i$  of type $(1,d_i,r+1)$;
    \item{}  $(L,W)$ is $\wa$-stable for any $\alpha > (r+1)\alpha_g$.
\end{enumerate}
\end{proposition}
\begin{proof}
(1) From the exact sequence \ref{EXSEQ:CI}, by tensoring with $L$ and passing to global sections we have: 
$$ 0 \to H^0( L|_{C_i^c}(-\Delta_i)) \to H^0(L) \xrightarrow{\rho_i} H^0(L_i) \to \cdots .$$
When we restrict $\rho_i$ to $W$ we obtain
\begin{equation}
\label{EXSEQ:Ri}
0 \to R_i  \to W  \to W_i \to 0.
\end{equation}
Then $W \simeq W_i$ if and only if  $R_i = \{ 0 \}$.

(2) The assertion follows from Theorem \ref{THM:stab}, since $(L_i,W_i)$ is $\alpha$-stable for any $\alpha > 0$. 
\end{proof}

Let $L \in \Pic^{(d_1,\cdots, d_{\gamma})}(C)$ be a line bundle with $d_i \geq 1$. 
Consider   a generated coherent system  $(L,W)$ of multitype $({\underline 1},d,r+1)$.
We can  associate to  it a coherent system $(E,V)$ of multitype $({\underline r},d,r+1)$  on $C$, with $E$  locally free.
As $ev_W:W\otimes\OO_C\to L$ is surjective it defines an exact sequence of vector bundles on $C$:
\begin{equation}
    \label{EXSEQ:DEFDUAL}
0 \to \Ker ev_W \to W \otimes \OO_C \to L \to 0,
\end{equation}
we will denote it  {\it the exact sequence defined by} $(L,W)$. 
Its dual gives the exact sequence
\begin{equation}
    \label{EXSEQ:DEF}
0 \to L^{-1} \to W^* \otimes \OO_C \to E \to 0,
\end{equation}
where we set $E = (\Ker ev_W)^*$.
Note that by construction $E$ is a vector bundle on $C$ of rank $r$ with determinant $\det(E)=L$. 
By taking the induced exact sequence on cohomology, as  $d_i\geq 1$, by Lemma \ref{LEM:chi} we have $h^0(L^{-1}) = 0$, so we have   an injective map 
$$0\to W^* \to H^0(E) \to \cdots$$
We define $V$ as the image of $W^*$ in $H^0(E)$ by the above inclusion.  As $V \simeq W^*$,  we have $\dim V= r+1$ . By the exactness of sequence \ref{EXSEQ:DEF}, we  also have  that $V$ generates $E$, so $(E,V)$ is a generated coherent system of multitype $({\underline r},d,r+1)$ on $C$. 
 
\begin{definition}
\label{defi:dualspan}
The coherent system $(E,V)$ is called  the dual span of $(L,W)$ and we will denote it as $D((L,W))$.
\end{definition} 

Vector bundles  arising as kernel of the evaluation map of a generated coherent system $(E,V)$,  on a smooth curve,  are called {\it kernel bundles} or {\it Lazarsfeld bundles}.  They were introduced by Butler and their stability  has    been  deeply studied by many authors.  For recent results on nodal curves with a node see \cite{BF3}.   
 
\begin{remark}
 Since $(L,W)$ is generated it defines a morphism $\Phi_{|W|}:C\to \PP W^*=\PP^r$. Consider the Euler sequence on $\PP W^*$:
$$0 \to \OO_{\PP^r}(-1) \to W^* \otimes \OO_{\PP^r} \to T_{\PP^r}(-1) \to 0.$$
By taking the pullback of this sequence with respect to $\Phi_{|W|}$ we obtain the exact sequence \ref{EXSEQ:DEF}. Hence we have $E=\Phi^{*}_{|W|}T_{\PP^r}(-1)$ and global sections of  $V$ are the pullback of the global sections of $T_{\PP^r}(-1)$ which are in $W^*$.
\end{remark}

\begin{remark}
\label{REM:Injective}
If $(E,V) = D((L,W)) = D((L',W'))$,  then $(L,W)=(L',W')$. In fact,  by the exact sequence defining $E$ we have that  $L=\det(E)=L'$.
We get the equality $W = W'$  by dualizing exact sequence \ref{EXSEQ:DEF} and considering the cohomology sequence. 
\end{remark}

\begin{proposition}
\label{PROP:2}
Let $C$ be a nodal curve as in Section \ref{nodal curves}, with $\gamma$ irreducible components.
Let $L  \in \Pic^{(d_1,\cdots, d_{\gamma})}(C)$ be   a line bundle  on the curve $C$, with $d_i \geq 1$.  Let  $(L,W)$ be a generated coherent system   of multitype $({\underline 1},d,r+1)$.   Consider its dual span $D((L,W))=(E,V)$ and its restriction $(E_i,V_i)$ to the component $C_i$. Then we have:
\begin{enumerate}
    \item if $R_i = 0$, then $(E_i,V_i)=D((L_i,W_i))$ and it is $\alpha$-stable for any $\alpha > (r-1)d_i$;
    \item if $R_i \not= 0$, then $(E_i,V_i)$ is $\alpha$-unstable for any $\alpha > 0$.
\end{enumerate}
\end{proposition}

\begin{proof}
If we tensor  the exact sequence \ref{EXSEQ:DEFDUAL} defined by $(L,W)$   with $\OO_{C_i}$, we  have again an  exact sequence too. Its relation with the exact sequence defined by the restriction $(L_i,W_i)$ is described in the following  diagram.
\begin{equation}
\label{EXDIAG:BIGEXDIAG}    
\xymatrix@R=8mm@C=8mm{
 & 
 R_i\otimes \OO_{C_i} \ar@{^{(}->}[d] \ar@{=}[r] &
 R_i\otimes \OO_{C_i} \ar@{^{(}->}[d] \\
0 \ar[r] &
(\Ker ev_W)\otimes \OO_{C_i} \ar@{^{(}->}[r] \ar@{->>}[d] & 
W\otimes \OO_{C_i} \ar@{->>}[r]^-{ev_{W}|_{C_i}} \ar@{->>}[d]^-{\rho_i} &
L_i \ar[r] \ar@{=}[d] &
0\\
0 \ar[r] &
\Ker ev_{W_i} \ar@{^{(}->}[r] & 
W_i\otimes \OO_{C_i} \ar@{->>}[r]^-{ev_{W_i}} &
L_i \ar[r] &
0
}
\end{equation}
where the second column is simply the exact sequence \ref{EXSEQ:Ri} tensored again with $\OO_{C_i}$. It is easy to see that the above diagram is indeed commutative. 
Finally, $\dim W_i \geq 2$ as $L_i$ is globally generated of degree $d_i\ge 1$ and $\Ker e_{W_i}$ is a vector bundle on $C_i$. 

If we dualize diagram \ref{EXDIAG:BIGEXDIAG} we can clearly see the relations between the dual span $(G_i,W_i^*)$ of $(L_i,W_i)$ and the restriction $(E_i,V_i)$ of the dual span $(E,V)$ of $(L,W)$. More precisely we have, as claimed 
$$(G_i,W_i^*)=(E_i,V_i) \Longleftrightarrow R_i=0.$$
Moreover, if   $R_i\neq 0$,  then  $(G_i,W_i^*)$ is a non trivial coherent subsystem of $(E_i,V_i)$.
\hfill\par
Now we prove the statements about $\alpha$-stability of $(E_i,V_i)$.

First of all, assume that $R_i=0$. Then we have $W_i \simeq W$ and, as we have seen, the restriction $(E_i,V_i)$ of $(E,V)$  is indeed the dual span of $(L_i,W_i)$. By \cite[Corollary 5.10]{BGMN},  it follows that  $(E_i,V_i)$  is  $\alpha$-stable for all $\alpha>d_i(r-1)$. 

Assume now that $R_i\neq 0$. Then $(G_i,W_i^*)$ is a non trivial coherent subsystem of $(E_i,V_i)$ of type $(s_i,d_i,s_i +1)$,  we prove that it  is a destabilizing subsystem of  $(E_i,V_i)$.
First of all note that since $s_i<r$ we have $\frac{(s_i+1)}{s_i} > \frac{(r+1)}{r}$. On the other hand, as $d_i\geq 1$, we have also $\frac{d_i}{s_i}>\frac{d_i}{r}$ hence
$$\mu_{\wa}(G_i,W_i^*)=\frac{d_i}{s_i}+\alpha\frac{s_i+1}{s_i}>\frac{d_i}{r}+\alpha\frac{r+1}{r}=\mu_{\wa}(E_i,V_i).$$
Hence, for all $\alpha>0$  we have that $(E_i,V_i)$ is $\alpha$-unstable as claimed.
\end{proof}

\begin{theorem}
\label{THM:1}
Let $C$ be a nodal curve as in Section \ref{nodal curves}, with $\gamma$ irreducible components and let ${\underline w}$ be a polarization on it. 
Let $L \in  \Pic^{(d_1,\cdots, d_{\gamma})}(C)$ be  a line bundle  on the curve $C$,  with $d_i \geq 1$. Let $(L,W)$ be  a generated coherent system of multitype $({\underline 1},d,r+1)$ satisfying $R_i = 0$ for $i = 1 \cdots \gamma$. 
Then the dual span $(E,V)$ of $(L,W)$ is $\wa$-stable for any $\alpha>(r+1){\alpha}_g$.
\end{theorem}
\begin{proof}
Since $R_i = 0$ for any $i= 1, \dots \gamma$, by Proposition \ref{Prop1}   we have that $(L_i,W_i)$ is a generated coherent system of type $(1,d_i,r+1)$. By Proposition \ref{PROP:2},  its dual span is $(E_i,V_i)$,  it is a coherent system of type $(r,d_i,r+1)$ and it  is $\alpha$-stable for $\alpha > (r-1)d_i$.  As $r$ and $r+1$  are coprime, we can apply Theorem \ref{THM:stab}  and  conclude the proof.
\end{proof}


\section{Moduli spaces of coherent systems of type $(r,d,r+1)$}

Let $C$ be  a nodal curve   as in Section \ref{nodal curves} with $\gamma$ irreducible components. Let $L \in \Pic^{(d_1,\cdots, d_{\gamma})}(C)$ be a globally generated  line bundle on $C$. For any $r \geq 1$ consider   the Grassmannian variety  $\G(r+1,H^0(L))$, parametrizing $(r+1)$-dimensional subspace of $H^0(L)$.  For any subspace $W \in \G(r+1, H^0(L))$,  $(L,W)$ is a coherent system of multitype $({\underline 1},d,r+1)$ on the curve $C$, with $d = \sum_{i=1}^{\gamma}d_i$.  
\begin{proposition}
\label{PROP:3}
Let $C$ be a  nodal curve  as in Section \ref{nodal curves} and let $\w$ be a polarization on it.  Let $r \geq 1$ and 
 let $L \in \Pic^{(d_1,\cdots, d_{\gamma})}(C)$ be a  line bundle on $C$ with $d_i\geq \max(2g_i+1,g_i+r)$. Then, for a general $W \in \G(r+1, H^0(L))$ and 
 for $\alpha >(r+1)\alpha_g$  we have that $(L,W) \in \Gwa'({\underline 1},d,r+1)$ and its dual span $(E,V) \in \Gwa'({\underline r},d,r+1)$.
\end{proposition}
\begin{proof}
By Lemma \ref{LEM:h1}, since $d_i \geq 2g_i+1$, $h^1(L)=0$ and $L$ is globally generated. Since $r\geq 1$ we have $\dim(W)\geq 2$. Hence, for $W$ general in $\G(r+1,H^0(L))$, $(L,W)$ is generated. 
By Lemma \ref{LEM:h1} $\rho_i \colon H^0(L) \to H^0(L_i)$ is surjective and its kernel has dimension $$\dim \Ker(\rho_i)=h^0(L|_{C_i^{c}}(-\Delta_i))=h^0(L)-h^0(L_i).$$ 
Since by assumption, $r+1\leq d_i-g_i+1=h^0(L_i)$, then for a general $W$ we have that $ R_i = W \cap \Ker (\rho_i) = \{ 0 \}$. Hence we can conclude using Proposition \ref{Prop1} and  Theorem \ref{THM:1}.
\end{proof}

As a corollary  we have  the following:
\begin{theorem}
\label{THM:2}
 Let $C$ be a nodal curve  as in Section \ref{nodal curves} with $\gamma$ components and let ${\underline w}$ be a polarization on it. 
 For any  integer $r \geq 1$ and for any  integer $d \geq \max(2p_a(C) + \gamma, p_a(C) + r \gamma)$  we have  $\Gwa'({\underline r},d,r+1) \not= \emptyset$ for any $\alpha  > (r+1)\alpha_g$.
\end{theorem}
\begin{proof}
It is enough to choose  a line bundle  $ L \in \Pic^{(d_1,\cdots, d_{\gamma})}(C)$ with $d_i\geq \max(2g_i+1,g_i+r)$ and $\sum_{i=1}^{\gamma}d_i = d$, then the assertion follows from Proposition \ref{PROP:3}. 
\end{proof}
Assume that $\Gwa'({\underline 1},d,r+1) \not= \emptyset$.  
For any $(d_1,\cdots,d_{\gamma}) \in {\mathbb Z}^\gamma$ with 
$\sum_{i=1}^{\gamma}d_i = d$,  
we define the following subscheme of $\Gwa'({\underline 1},d,r+1)$
\begin{equation}
\label{subschemeX}
X_{d_1,\dots,d_{\gamma}} = \{  (L,W) \in \Gwa'({\underline 1},d,r+1) \ \  \vert \ \  L_i \in \Pic^{d_i}(C_i) \}, 
\end{equation}
and consider its closure $\overline{X_{d_1,\dots,d_{\gamma}}} $ in $\Gwa({\underline 1},d,r+1)$.
Then we have the following:

\begin{theorem}
\label{Main1}
 Let $C$ be a nodal curve  as in Section \ref{nodal curves} with $\gamma$ components and let ${\underline w}$ be a polarization on it.
 Let $r \geq 1$ and $d_i \geq  \max(2g_i+1,g_i+r)$. 
 Then,  for any $\alpha > (r+1)\alpha_g$,   $\overline{X_{d_1,\dots,d_{\gamma}}} $ is an irreducible component of $\Gwa({\underline 1},d,r+1)$ of dimension  equal to the Brill-Noether number
 $$\beta_C(1,d,r+1)=  p_a(C)+(r+1)(d-r - p_a(C)).$$
 Moreover, each $(L,W) \in X_{d_1,\cdots,d_{\gamma}}$ is a smooth point of the moduli space. 
\end{theorem}
\begin{proof}
Since we assumed $d_i \geq \max(2g_i+1,g_i+r)$, then by Proposition \ref{PROP:3} we have that 
$X_{d_1,\dots,d_{\gamma}} \not= \emptyset$. Let $(L,W) \in X_{d_1,\dots,d_{\gamma}}$. Then as we have seen before, we have $h^1(L) = 0$, so the Petri map
$$\mu_{L,W} \colon W \otimes H^0({\omega}_C \otimes L^*) \to H^0({\omega}_C \otimes L \otimes L^*) $$
is injective. By Theorem \ref{Petritheorem},  any $(L,W) \in X_{d_1,\cdots,d_{\gamma}} $ is a smooth point of the moduli space $\Gwa(1,d,r+1)$
and the dimension of $\Gwa(1,d,r+1)$ at $(L,W)$ is given by the Brill-Noether number
$$\beta_C(1,d,r+1) = p_a(C) + (r+1)(d-r-p_a(C)).$$
In order to prove the assertion,  we can consider the  natural  morphism
$$\pi \colon X_{d_1,\dots,d_{\gamma}} \to \Pic^{d_1}(C_1) \times \dots \times \Pic^{d_{\gamma}}(C_{\gamma}),$$
sending $(L,W) \to (L_1,\dots,L_{\gamma}).$
We recall that, since $C$ is a curve of compact type, we have an isomorphism 
$$\Pic^{d_1}(C_1) \times \cdots \times  \Pic^{d_{\gamma}}(C_{\gamma}) \simeq \Pic^{(d_1,\cdots,d_{\gamma})}(C).$$
By Proposition \ref{PROP:3}, $\pi$ is surjective and  each fiber 
$\pi^{-1}(L_1,\cdots,L_{\gamma})$ is an open subset of the Grassmannian variety  $\Gr(r+1,H^0(L))$, where $L \in \Pic^{(d_1,\cdots,d_{\gamma})}(C)$ is the unique  line bundle on the curve $C$  corresponding to $(L_1,\cdots,L_{\gamma})$. Hence all the fibers of $\pi$ are  irreducible  and equidimensional of dimension 
$(r+1)(h^0(L) -r-1)= (r+1)(d-r-p_a(C))$.

We will denote by $Z_i$ the irreducible components of $X_{d_1,\cdots,d_{\gamma}}$.  Since all fiber are irreducible and equidimensional,
we have
$$Z_j= \bigcup {\pi}^{-1}(L_1,\cdots,L_{\gamma}), \qquad (L_1,\cdots,L_{\gamma}) \in \pi(Z_j)$$
so $\dim Z_j = \dim \pi(Z_j) + (r+1)(d-r-p_a(C))$.
\hfill\par
Since $\pi$ is surjective, 
there is an irreducible component $Z_0$ of $X_{d_1,\cdots,d_{\gamma}}$ such that the restriction  $\pi\vert_{Z_0}$  is dominant. Hence we have 
$$ \dim Z_0= p_a(C) + (r+1)(d-r-p_a(C)) = \beta_C(1,d,r+1). $$
 This implies that the closure $\overline{Z_0}$
is an irreducible component of $\Gwa({\underline 1},d,r+1)$. 
\hfill\par
In order to prove that $X_{d_1,\cdots,d_{\gamma}}$ is a component of the moduli space as claimed, it is enough to show that $Z_0 = X_{d_1,\cdots,d_{\gamma}}$. First of all we will prove that  for any other possible irreducible component $Z_j$ of $X_{d_1,\dots,d_{\gamma}}$  we have $\dim Z_j < \dim Z_0$. 
Indeed, otherwise, $\overline{Z_j}$ would be an irreducible component of the moduli space $\Gwa({\underline 1},d,r+1)$, moreover since  $\pi(Z_j) $ would be an open subset of  $\Pic^{d_1}(C_1) \times \dots \times \Pic^{d_{\gamma}}(C_{\gamma})$ we would have  $Z_j \cap Z_0 \not= \emptyset$. This is impossible since all  points of 
$Z_j$ and $Z_0$ are smooth points of the moduli space and thus they cannot be common to two irreducible components.
\hfill\par
Assume that $Z_j$ is an irreducible component of 
$X_{d_1,\cdots,d_{\gamma}}$ different from $Z_0$. Let $Y_j$ be the unique irreducible component of $\Gwa({\underline 1},d,r+1)$ containing $Z_j$. By construction the dimension of $Y_j$ is equal to the Brill-Noether number (as it contains $Z_j$ whose points are smooth for the moduli space) and  it is strictly bigger then the dimension of $Z_j$. By the same argument as above we have $Y_j \cap Z_0 = \emptyset$. 

Consider the intersection $U_j=Y_j \cap \Gwa'({\underline 1},d,r+1)$. It is a non empty open subset of $Y_j$ since it contains $Z_j$. 
Then the generic point of $U_j$ cannot lie on $X_{d_1,\dots, d_{\gamma}}$ as $U_j$ is disjoint from $Z_0$. This implies that $U_j$ contains coherent systems with different multidegrees which is impossible as $Y_j$ is irreducible (see Remark \ref{families}).     
\end{proof}
 
By the dual span construction we obtain the following result:

\begin{theorem}
\label{main2}
 Let $C$ be a nodal curve  as in Section \ref{nodal curves} with $\gamma$ components and let ${\underline w}$ be a polarization on it.
Let $r \geq 1$,  $d_i \geq  \max(2g_i+1,g_i+r)$ and set  $d= \sum_{i=1}^{\gamma} d_i$.  
 Then,  for any $\alpha > (r+1)\alpha_g$,   there exists  an irreducible component  $Y_{d_1,\cdots, d_{\gamma}}$ of $\Gwa({\underline r},d,r+1)$  which is birational to $X_{d_1,\dots,d_{\gamma}}$, of dimension  equal to the Brill-Noether number
 $$\beta_C(r,d,r+1)=  p_a(C)+(r+1)(d-r - p_a(C)).$$
Moreover,  a general $(E,V) \in Y_{d_1,\cdots,d_{\gamma}}$ satisfies  the following properties:
 \begin{enumerate}
 \item{}  it is a generated coherent system  with $E$ locally free and $\deg (E_i) = d_i$, $i= 1,\cdots, \gamma;$
 \item{} it is a smooth   point of the moduli space $\Gwa({\underline r},d,r+1)$;
 \item{} for any $i = 1, \cdots, \gamma$ the restriction $(E_i,V_i)$ is a generated  coherent system  on $C_i$ of type $(r,d_i,r+1)$ which is $\alpha$-stable for any $\alpha > (r-1)d_i$.
 \end{enumerate}
 \end{theorem}
\begin{proof}
We consider the irreducible component 
$\overline{X_{d_1,\cdots,d_{\gamma}}}$ of $\Gwa({\underline 1},d,r+1)$ 
described in Theorem \ref{Main1}. 
Let $(L,W) \in X_{d_1,\cdots,d_{\gamma}}$, assume that it is  a generated  coherent system, then  we can define its dual span $D((L,W))=(E,V)$, see Definition \ref{defi:dualspan}.  It is a generated  coherent system of multitype 
$({\underline r},d,r+1)$ and  $E$ is  locally  free.  If moreover, $R_i= 0$ for any $i= 1,\dots,\gamma$, then by Theorem \ref{THM:1},    it follows that 
$(E,V) \in  \Gwa({\underline r},d,r+1)$, for any $\alpha > (r+1)\alpha_g$.  This allows us to define a map
$$
\xymatrix{
\mathcal{D} : \overline{X_{d_1,\cdots,d_{\gamma}}} \ar@{-->}[r] & \Gwa({\underline r},d,r+1),
}$$
sending $(L,W) \to D((L,W))$. 
Actually,   $\mathcal D$ is defined on the   subset
$$\mathcal U = \{ (L,W) \in X_{d_1,\cdots, d_{\gamma}} \vert \ (L,W) \mbox{ generated and } R_i = 0, \ i = 1,\dots,\gamma \},$$
which is a non empty open subset by Proposition \ref{PROP:3}. 
We will prove that the restriction $\mathcal{D}_{\vert \mathcal U}$ is a birational morphism onto its image $\mathcal{D}({\mathcal U})$. 
Let $({\mathcal L},{\mathcal W},\xi)$ be a family of coherent systems of ${\mathcal U}$ parametrized by a a connected scheme $T$. Then,  
$\mathcal L$ is a locally free sheaf on $C \times T$ such that ${\mathcal L}_{\vert C_i \times t}  \in \Pic^{d_i}(C_i)$. Let $\pi \colon C \times T \to T$ be the projection,  $\mathcal W$ is a locally free sheaf on $T$  and $\xi \colon \pi^* \mathcal W \to \mathcal L$ is a map of locally free sheaves such that for any $t \in T$ the map
$$ \xi_t : W_t \otimes \OO_{C \times t} \to L_t$$
induces the following injective map $H^0(\xi_t) \colon W_t \to H^0(L_t)$.
Since $(L_t,W_t)$ is generated, $\xi_t$ is surjective for any $t \in T$,  then it follows that  $\xi$ is a surjective map of locally free sheaves on $C \times T$. We can consider its kernel $\Ker \xi$, it is a locally free sheaf on $C \times T$. We have then an exact sequence of locally free sheaves on $C\times T$
$$0 \to \Ker(\xi)\xrightarrow{\eta}\pi^*\mathcal{W}\xrightarrow{\xi}\mathcal{L}\to 0.$$
and its dual
$$0 \to\mathcal{L}^* \xrightarrow{\xi^*}\pi^*\mathcal{W}^*\xrightarrow{\eta^*} \mathcal E\to 0,$$
where we have denoted by $\mathcal{E}$ the sheaf $\Ker(\xi)^*$. This implies that, for all $t\in T$ the map
$$\eta^*_t:W_t^* \otimes \OO_{C \times t} \to E_t$$ 
is surjective and the map 
$H^0(\eta_t^*) \colon W_t^* \to H^0(E_t)$ is injective. 
This implies that  $({\mathcal E}, {\pi^*\mathcal W}^*, \eta^*)$ is a family of generated coherent systems of multitype $({\underline r},d,r+1)$, which are $\wa$-stable for any $\alpha > (r+1)\alpha_g.$
This ensure that the map $\mathcal{D}|_{\mathcal U}$ is a morphism and it is injective by construction, see remark \ref{REM:Injective}.  This proves that $\mathcal{D}$ is a birational map onto its image.  
We denote by $Y_{d_1,\cdots,d_{\gamma}}$ the closure of $\mathcal{D}({\mathcal U})$ in $\Gwa({\underline r},d,r+1)$. 
It is an irreducible subscheme of 
$\Gwa({\underline r},d,r+1)$ of dimension $\beta_C(1,d,r+1) = p_a(C) + (r+1)(d-r-p_a(C))$. 
Note that we have
$$\beta_C(1,d,r+1) = \beta_C(r,d,r+1),$$
so, in order to prove the assertion, we will   show  that  for each coherent system $(E,V) \in {\mathcal D}({\mathcal U})$ the Petri map 
$$\mu_{E,V} \colon V \otimes H^0({\omega}_C \otimes E^*) \to H^0({\omega}_C \otimes E \otimes E^*)$$
is injective.
Consider the exact sequence defining $(E,V)$, i.e.
$$ 0 \to L^{-1} \to V \otimes O_C \to E \to 0$$
and tensor it with $L$. This yields a surjective map $V\otimes H^1(L)\to H^1(E\otimes L)$. Under our assumptions we have $h^1(L)=0$ by Proposition \ref{PROP:3} so  $H^1(E\otimes L)=0$ too. In particular, $H^0(\omega_C\otimes E^*\otimes L^{-1})=0$ by Serre duality.
If we consider again the above exact sequence and tensor it with $E^*\otimes \omega_C$ and take cohomology we obtain
$$0\to H^0(\omega_C\otimes E^*\otimes L^{-1})\to V\otimes H^0(\omega_C\otimes E^*)\xrightarrow{\mu_{E,V}} H^0(\omega_C\otimes E\otimes E^*) \to \cdots$$
which implies that $\mu_{E,V}$ is indeed injective as claimed.
\end{proof}

In what follows we will denote by $\mathcal{G}_{C_i,\alpha}(s,d_i,r+1)$ the moduli space of $\alpha$-stable coherent systems of type $(s,d_i,r+1)$ on the curve $C_i$. Assume that it is not empty and denote by $D_i$ the map
$$
\xymatrix{
\mathcal{G}_{C_i,\alpha}(1,d_i,r+1)\ar@{-->}[r]^-{D_i} & \mathcal{G}_{C_i,\alpha}(r,d_i,r+1).
}
$$
sending a generated coherent system $(L_i,W_i)$ to its dual span $D_i((L_i,W_i))$. The map $D_i$ is birational, see \cite[Corollary 5.10]{BGMN}. 

Let $\overline{X_{d_1,\cdots,d_{\gamma}}}$ and $Y_{d_1,\dots,d_\gamma}$ be the irreducible components of $\Gwa(\underline{1},d,r+1)$ and $\Gwa(\underline{1},d,r+1)$ respectively described in 
Theorems \ref{Main1} and \ref{main2}. We can consider the diagram
\begin{equation}
\label{DIAG:D}
\xymatrix@C=50pt{
\overline{X_{d_1,\cdots,d_{\gamma}}}\ar@{-->}[r]^-{\mathcal{D}}\ar@{-->}[d]_{\pi_1} &     Y_{d_1,\dots,d_\gamma}\ar@{-->}[d]^{\pi_2} \\
  \Pi_{i=1}^{\gamma} \mathcal{G}_{C_i,\alpha}(1,d_i,r+1)\ar@{-->}[r]_-{\Pi D_i} & 
  \Pi_{i=1}^{\gamma} \mathcal{G}_{C_i,\alpha}(r,d_i,r+1)
}
\end{equation}
where the vertical maps  ${\pi_1}$  and ${\pi_2}$ are  restrictions to the components of the curve $C$.

Then we have the following result:

\begin{theorem}
\label{THM:B}
In the hypothesis of Theorem \ref{main2}. 
The diagram \ref{DIAG:D} is commutative and the maps ${\pi_1}$  and ${\pi_2}$ are both dominant. Moreover, the general fiber of $\pi_i$ has dimension $\delta\cdot r(r+1)$,  where $\delta$ denotes the number of nodes on the curve $C$. 
\end{theorem}
\begin{proof}
Let $\mathcal U \subset  X_{d_1,\cdots,d_{\gamma}}$ be the open subset where $\mathcal D$ is defined. 
Note that if $(L,W) \in \mathcal U$, then then  by Theorems \ref{Main1} and \ref{main2}, we have:  $D((L,W))= (E,V) \in Y_{d_1,\cdots,d_{\gamma}}$,    the restrictions $(E_i,V_i)$ are  $\alpha$-stable coherent systems of type $(r,d_i,r+1)$ for any $\alpha >(r-1)d_i$, $(L_i,W_i)$ are coherent systems of type $(1,d_i,r+1)$   and we have $D_i((L_i,W_i)) = (E_i,V_i)$. 
Hence, the diagram commutes  and in particular we have that both $\mathcal{G}_{C_i,\alpha}(1,d_i,r+1)$ and $\mathcal{G}_{C_i,\alpha}(r,d_i,r+1)$ are non empty. 
\hfill\par
Since $D_i$ is birational,  see \cite{BGMN}, we have:
$$\dim \mathcal{G}_{C_i,\alpha}(r,d_i,r+1)=  \dim \mathcal{G}_{C_i,\alpha}(1,d_i,r+1) = \beta_{C_i}(1,d_i,r+1)= g_i + (r+1)(d_i-g_i-r).$$
This implies:
\begin{multline*}
\dim \Pi_{i=1}^{\gamma} \mathcal{G}_{C_i,\alpha}(1,d_i,r+1) = \dim \Pi_{i=1}^{\gamma} \mathcal{G}_{C_i,\alpha}(r,d_i,r+1)=\\
=p_a(C) + (r+1)(d-p_a(C) -r) - (r+1)r\delta
\end{multline*}
where $\delta = \gamma -1, $
 denotes the number of nodes of $C$. 
\hfill\par
We will prove  that $\pi_1$ is dominant. Since 
it is a  rational map  between two irreducible  varieties  and  $\dim X_{d_1,\cdots,d_{\gamma}} - \dim  \Pi_{i=1}^{\gamma} \mathcal{G}_{C_i,\alpha}(1,d_i,r+1)= r(r+1)\delta$, then  it is enough to show that a general fiber has dimension $r(r+1)\delta$. 
Let $(L,W)\in \mathcal U$. We will compute the dimension of the fiber $\mathcal F$ of $\pi_1$ over $\pi_1((L,W))$. Let  $(L_i,W_i)$  be the restrictions of $(L,W)$ to the components of $C$. Since there is a unique line bundle on the curve $C$ having restrictions $L_i$, then we have:
$$\mathcal F = \{ (L,W') \in \mathcal U \quad  \vert\quad  {W'}|_{C_i} = W_i \}.$$
Note that $\mathcal F \not= \emptyset$ since $(L,W) \in {\mathcal F}$. 
We recall that we have an exact sequence as follows, see Lemma \ref{LEM:chi}:
$$ 0 \to L \to \bigoplus_{i=1}^{\gamma}L_i \to {T} \to 0$$
where ${T}=\bigoplus_{j=1}^{\delta}\mathbb{C}_{p_{j}}$. As $h^1(L)=0$ we also have the exact sequence
$$0\to H^0(L)\to \bigoplus_{i=1}^{\gamma}H^0(L_i)\xrightarrow{\alpha} \mathbb{C}^{\delta}\to 0,$$
where $\alpha$  at the node $p_{j} = C_{j1} \cap C_{j2}$ is the map sending 
$(s_1,\cdots,s_{\gamma}) \to s_{j1}(p_j)-s_{j2}(p_j)$.
We can consider the restriction $\alpha'$
 of $\alpha$ to $\bigoplus_{i=1}^{\gamma} W_i$.  It is a surjetive map, since $(L_i,W_i)$ is generated. So we have: 
$$0\to S \to \bigoplus_{i=1}^{\gamma}W_i\xrightarrow{\alpha'} \mathbb{C}^{\delta}\to 0$$
and $\dim S =\sum_{i=1}^{\gamma} W_i - \delta = \delta r + r+1$.
So we have:
$$ \mathcal F \simeq  \{ W'\in \Gr(r+1,S) \ \vert  \  (L,W') \in \mathcal U \}.$$
Since  by Proposition  \ref{PROP:3} the subset 
$\{ W'  \in \Gr(r+1,H^0(L)) \,|\,  (L,W') \in \mathcal{U} \}$
is  a non empty  open subset  of the variety $\Gr(r+1,H^0(L))$, then $\mathcal F$  is a non empty open subset of $\Gr(r+1,S)$. Hence $\dim \mathcal F = \dim \Gr(r+1,S) = \delta r(r+1)$ as claimed. This concludes the proof. 
\end{proof}

\end{document}